\documentclass[11pt]{amsart}
\usepackage{amssymb}
\usepackage{latexsym}
\usepackage{graphics}
\usepackage{graphicx}
\usepackage{epstopdf}
\usepackage[all]{xy}
\usepackage{verbatim}
\input epsf

\def\TSG{{\mathrm{TSG}_+}}
\def\fix{{\mathrm{fix}}}
\def\Aut{{\mathrm{Aut}}}

\newtheorem{prop}{Proposition}
\newtheorem{fact}{Fact}

\newtheorem{thm}{Theorem}

\newtheorem{lemma}{Lemma}

\newtheorem*{edge}{Edge Embedding Lemma}
\newtheorem*{subgroup}{Subgroup Theorem}
\newtheorem*{subgroupcor}{Subgroup Corollary}

\newtheorem*{automorphism}{Automorphism Theorem}
\newtheorem*{isometry}{Isometry Theorem}
\newtheorem*{A4}{Bipartite $A_4$ Theorem}
\newtheorem*{S4}{Bipartite $S_4$ Theorem}
\newtheorem*{A5}{Bipartite $A_5$ Theorem}
\newtheorem*{fixed vertex}{Fixed Vertex Property}
\def\inv{{^{-1}}}

\def\Z{{\mathbb Z}}

\def\g{{\gamma}}

\def\f{{\phi}}

\def\TSG{{\mathrm{TSG_+}}}
\def\Aut{{\mathrm{Aut}}}
\def\Diff{{\mathrm{Diff_+}}}
\def\fix{{\mathrm{fix}}}

\def\so{{\mathrm{SO}}}

\newcommand{\orb}[1]{\langle #1 \rangle}

\begin{document}
\title[Bipartite Polyhedral]{Complete bipartite graphs whose topological symmetry groups are polyhedral}
\author{Blake Mellor}



\address{Department of Mathematics, Loyola Marymount University, Los Angeles, CA 90045, USA}

\date \today

\thanks{This research was supported in part by NSF Grant DMS-0905687.}

\begin{abstract}

We determine for which $n$, the complete bipartite graph $K_{n,n}$ has an embedding in $S^3$ whose topological symmetry group is isomorphic to one of the polyhedral groups: $A_4$, $A_5$, or $S_4$.

\end{abstract}

\maketitle

\section{Introduction}

The {\em topological symmetry group} was introduced by Jon Simon \cite{si} to study symmetries of polymers and other flexible molecules for which the traditional point group (i.e. the group of rigid symmetries) is insufficient.  We consider an abstract graph $\gamma$ with automorphism group $\Aut(\gamma)$, and let $\Gamma$ be an embedding of $\gamma$ in $S^3$.  The {\em topological symmetry group} of $\Gamma$, denoted $\mathrm{TSG}(\Gamma)$, is the subgroup of $\Aut(\gamma)$ induced by homeomorphisms of the pair $(S^3, \Gamma)$.  The {\em orientation preserving topological symmetry group} of $\Gamma$, denoted $\TSG(\Gamma)$, is the subgroup of $\Aut(\gamma)$ induced by by orientation preserving homeomorphisms of the pair $(S^3, \Gamma)$.  In this paper we are only concerned with $\TSG(\Gamma)$, so we will refer to it as simply the {\em topological symmetry group}.

It has long been known that every finite group can be realized as $\Aut(\gamma)$ for some graph $\gamma$ \cite{fr}.  However, this is {\em not} true for topological symmetry groups.  Results of Flapan, Naimi, Pommersheim and Tamvakis \cite{fnpt}, in combination with the recent proof of the Geometrization Conjecture \cite{mf}, show that a group is the topological symmetry group for some embedding of a 3-connected graph if and only if it is a finite subgroup of $\so(4)$.  However, their results do not give any information as to which graphs can be used to realize any particular group.  The first results along these lines have been for the family of complete graphs $K_n$.  Flapan, Naimi and Tamvakis \cite{fnt} classified the groups which could be realized as the topological symmetry group for an embedding of a complete graph; subsequently, Flapan, Naimi and the author determined exactly which complete graphs had embeddings that realized each group. \cite{fmn2, fmn3}.

In this paper we turn to another well-known family of graphs, the {\it complete bipartite graphs} $K_{n,n}$.  Unlike the complete graphs, where only some of the subgroups of $\so(4)$ are realizable as topological symmetry groups, {\em any} finite subgroup of $\so(4)$ can be realized as the topological symmetry group of an embedding of some $K_{n,n}$ \cite{fnpt}.  So the complete bipartite graphs are a natural family of graphs to investigate in order to better understand the full range of possible topological symmetry groups.  The finite subgroups of $\so(4)$ have been classified and they can all be described as quotients of products of cyclic groups $\Z_n$, dihedral groups $D_n$, and the symmetry groups of the regular polyhedra ($A_4$, $S_4$ and $A_5$) \cite{du}.  As a first step in understanding the topological symmetry groups of complete bipartite graphs, it makes sense to first consider these building blocks.  In this paper, we will consider the polyhedral groups $A_4$, $S_4$ and $A_5$; in later work we will consider cyclic and dihedral groups \cite{hmp}.  We will prove the following three theorems:

\begin{A4}
A complete bipartite graph $K_{n,n}$ has an embedding $\Gamma$ in $S^3$ such that  $\TSG(\Gamma) \cong A_4$ if and only if $n \equiv 0$, $2$, $4$, $6$, $8 \pmod {12}$ and $n \geq 4$.
\end{A4}

\begin{S4}
A complete graph $K_{n,n}$ has an embedding $\Gamma$ in $S^3$ such that  $\TSG(\Gamma) \cong S_4$ if and only if $n \equiv 0$, $2$, $4$, $6$, $8 \pmod {12}$, $n \geq 4$ and $n \neq 6$.
\end{S4}

\begin{A5}
A complete graph $K_{n,n}$ has an embedding $\Gamma$ in $S^3$ such that $\TSG(\Gamma) \cong A_5$ if and only if $n \equiv 0$, $2$, $12$, $20$, $30$, $32$, $42$, $50 \pmod{60}$ and $n > 30$.
\end{A5}

\section{Necessity of the conditions}

In this section we will prove the necessity of the conditions given in the Bipartite $A_4$, $S_4$ and $A_5$ Theorems.  We begin with a few results proved elsewhere which we will find useful.  The first is the observation that conjugate elements of $A_4$, $S_4$ or $A_5$ fix the same number of vertices of a graph.  We will frequently refer to the elements of these groups of order 2 as {\it involutions}.

\begin{fixed vertex} \cite{fmn2}
Let $G\cong A_4$, $A_5$ and suppose $G$ acts faithfully on a graph $\Gamma$.  Then all elements of $G$ of a given order fix the same number of vertices.  Furthermore, since all of the non-trivial elements of $G$ have prime order, all of the elements in a given cyclic subgroup fix the same vertices.  

Let $H$ be isomorphic to $S_4$ and suppose that $H$ acts faithfully on $\Gamma$.  Then all elements of $H$ of order 3 fix the same number of vertices, and all elements of $H$ of order 4 fix the same number of vertices.  All involutions of $H$ which are in $G\cong A_4$ fix the same number of vertices, and all involutions of $H$ which are not in $G$ fix the same number of vertices.
\end{fixed vertex}

The following result allows us to assume that the topological symmetry group is induced by elements of $\so(4)$ -- namely, rotations and glide rotations.  It follows from \cite{fnpt}, together with the recent proof of the Geometrization Theorem \cite{mf}.

\begin{isometry} \cite{fnpt}
Let $\Gamma$ be an embedded 3-connected graph, and let $H = \TSG(\Gamma)$.  Then $\Gamma$ can be re-embedded as $\Gamma'$ such that $H \leq \TSG(\Gamma')$ and $\TSG(\Gamma')$ is induced by an isomorphic subgroup of $\so(4)$.
\end{isometry}

From the Isometry Theorem, we may assume that the elements of the topological symmetry group are induced by rotations (with fixed point set a geodesic circle) and glide rotations (with no fixed points).  The fixed point sets of two rotations are hence either disjoint, intersect in exactly two points, or are identical.

We now look specifically at complete bipartite graphs $K_{n,n}$.  Recall that the graph $K_{n,n}$ consists of two vertex sets $V$ and $W$, each containing $n$ vertices, such that every vertex in $V$ is adjacent to every vertex in $W$; but no two vertices of $V$ are adjacent, and no two vertices of $W$ are adjacent.  The following observation describes the automorphisms of a complete bipartite graph $K_{n,n}$.

\begin{fact} \label{L:automorphisms}
Let $V$ and $W$ be the vertex sets of $K_{n,n}$, and let $\f: V \cup W \rightarrow V \cup W$ be a permutation of the vertices.  Then $\f$ is an automorphism of the graph if and only if either $\f(V) = V$ or $\f(V) = W$.
\end{fact}

The next theorem classifies the automorphisms of $K_{n,n}$ that can be induced by orientation-preserving homeomorphisms of $S^3$.

\begin{automorphism} \cite{flmpv}
Let $n>2$ and let $\varphi$ be an order $r$ automorphism of a complete bipartite graph $K_{n,n}$ with vertex sets $V$ and $W$.  There is an embedding $\Gamma$ of $K_{n,n}$ in $S^3$ with an orientation preserving homeomorphism $h$ of $(S^3,\Gamma)$ inducing $\varphi$ if and only if all vertices are in $r$-cycles except for the fixed vertices and exceptional cycles explicitly mentioned below (up to interchanging $V$ and $W$):

\begin{enumerate} 
\item There are no fixed vertices or exceptional cycles.

\item $V$ contains $mr$ fixed vertices for some positive integer $m$.

\item $V$ and $W$ each contain 1 fixed vertex or each contain 2 fixed vertices.

\item $j|r$ and $V$ contains some $j$-cycles.

\item $r=\mathrm{lcm}(j,k)$, and $V$ contains some $j$-cycles and $k$-cycles.

\item $r=\mathrm{lcm}(j,k)$, and $V$ contains some $j$-cycles and $W$ contains some $k$-cycles.

\item $V$ and $W$ each contain one 2-cycle.

\item $\frac{r}{2}$ is odd, $V$ and $W$ each contain one 2-cycle, and $V$ contains some $\frac{r}{2}$-cycles.

\item $\varphi(V)=W$ and $V\cup W$ contains one 4-cycle.

\end{enumerate}

\end{automorphism}

The following lemmas will place further restrictions on the automorphisms which are contained in a polyhedral subgroup.

\begin{lemma} \label{L:f(V)=V}
Let $\Gamma$ be an embedding of $K_{n,n}$ such that $G \leq \TSG(\Gamma)$, where $G$ is isomorphic to $A_4$ or $A_5$.  Then every element of $G$ fixes $V$ (and hence $W$) setwise.
\end{lemma}
\begin{proof}
By Fact \ref{L:automorphisms}, any element of odd order must fix $V$ setwise.  The only elements of $G$ of even order are involutions of order 2.  Recall that $A_4 = \langle \f, g : \f^2 = g^3 = (\f g)^3 = 1 \rangle$ and $A_5 = \langle \f, g : \f^2 = g^3 = (\f g)^5 = 1 \rangle$.  So if $\f \in G$ is an involution, then there an element $g \in G$ of order 3 such that $\f g$ has odd order.  So $g(V) = V$ and $\f g(V) = V$.  Hence $\f(V) = \f (g(V)) = V$.  So the involutions also fix $V$ setwise.
\end{proof}

\begin{lemma} \label{L:n2}
Let $G\leq \Aut(K_{n,n})$ which is isomorphic to $A_4$.  Suppose there is an embedding $\Gamma$ of $K_{n,n}$ in $S^3$ such that $G$ is induced on $\Gamma$ by an isomorphic subgroup $\widehat{G}\leq\so(4)$.   Then an order 2 element of $G$ cannot fix exactly one vertex of $V$.
\end{lemma}
\begin{proof}
Suppose there is an involution which fixes exactly one vertex of $V$.  Then by the Automorphism Theorem, the involution also fixes one vertex of $W$, and hence the edge between them.  Since every element of $A_4$ fixes $V$ and $W$ setwise, the Fixed Vertex Property implies all involutions in $A_4$ fix the same number of vertices in $V$ and $W$, so each involution fixes exactly one edge and hence is induced by a rotation in $\widehat{G}$.  Two distinct involutions can't fix the same edge (or they would correspond to the same rotation in $\widehat{G}$), so there are three edges $e_1 = \overline{v_1w_1}$, $e_2 = \overline{v_2w_2}$, $e_3 = \overline{v_3w_3}$, each fixed pointwise by one involution.  Say that $\f_i$ is the involution fixing $e_i$ pointwise.  The involutions $\f_i$ all commute in $A_4$, so $\f_i(\f_j(v_i)) = \f_j(\f_i(v_i)) = \f_j(v_i)$.  So $\f_j(v_i)$ is fixed by $\f_i$.  But $\f_j(v_i) \in V$ by Lemma \ref{L:f(V)=V}, and $\f_i$ only fixes one vertex in $V$, so $\f_j(v_i) = v_i$.  Therefore, the three edges $e_1, e_2, e_3$ are all fixed pointwise by all three involutions.  This means the involutions are not distinct, which is a contradiction.
\end{proof}

\begin{lemma} \label{L:n2=4}
Let $G\leq \Aut(K_{n,n})$ which is isomorphic to $D_2 \cong \Z_2 \times \Z_2$.  Suppose there is an embedding $\Gamma$ of $K_{n,n}$ in $S^3$ such that $G$ is induced on $\Gamma$ by an isomorphic subgroup $\widehat{G}\leq\so(4)$.   Assume the involutions of $G$ each fix exactly two vertices of $V$ and two vertices of $W$.  Then the axes of the corresponding involutions in $\widehat{G}$ are concurrent at two points in $V$ (or two points in $W$).
\end{lemma}
\begin{proof}
Let $\f_i$, $1 \leq i \leq 3$, denote the three involutions.  Each involution fixes 2 vertices in each of $V$ and $W$, and hence fixes four edges.  Say that $\f_i$ fixes $\{v_i, v_i', w_i, w_i'\}$.  Since the involutions commute in $D_2$, each involution fixes the fixed point set of the others (setwise); also, by Fact \ref{L:automorphisms}, they fix $V$ and $W$ setwise.  So $\f_j(v_i)$ is either $v_i$ or $v_i'$.  Note that $\f_j$ cannot fix both $v_i$ and $w_i$, since distinct involutions cannot fix the same edge (or they would correspond to the same rotation in $\widehat{G}$).

Without loss of generality, assume $\f_2$ fixes neither $v_1$ nor $w_1$.  So $\f_2(v_1) = v_1'$ and $\f_2(w_1) = w_1'$.  Then the fixed point sets of $\widehat{\f_1}$ and $\widehat{\f_2}$ do not meet in any vertices of the graph.  But, since each involution fixes a subgraph homeomorphic to $S^1$, and edges of the embedding cannot intersect in their interiors, this means the fixed point sets of $\widehat{\f_1}$ and $\widehat{\f_2}$ are disjoint circles.  Also, since the involutions commute, $\widehat{\f_1}$ and $\widehat{\f_2}$ each preserve both of these circles setwise, so the two circles are linked.  Hence $\widehat{\f_3} = \widehat{\f_1}\widehat{\f_2}$ is the composition of two rotations around linked circles which preserves both circles setwise, and is therefore a glide rotation.  But this implies $\widehat{\f_3}$ has no fixed point set, which is a contradiction.  

So $\f_2$ must fix exactly one of $v_1$ and $w_1$.  Without loss of generality, say $\f_2$ fixes $v_1$ (and hence also $v_1'$).  Then $\f_3(v_1) = \f_2\f_1(v_1) = \f_2(v_1) = v_1$, so $\f_3$ also fixes $v_1$ and $v_1'$.  This completes the proof.
\end{proof}

\medskip
\noindent{\sc Remark:}  In the remainder of the paper we will assume the axes are concurrent at two points of $V$, whenever we can do so without loss of generality.

\begin{lemma} \label{L:n3>2}
Let $G\leq \Aut(K_{n,n})$ which is isomorphic to $A_4$ or $A_5$.  Suppose there is an embedding $\Gamma$ of $K_{n,n}$ in $S^3$ such that $G$ is induced on $\Gamma$ by an isomorphic subgroup $\widehat{G}\leq\so(4)$.   Assume the involutions of $G$ each fix exactly two vertices of $V$ and two vertices of $W$.  Then there are at least two vertices in $V$ (or two in $W$) that are fixed by every element of $G$.  
\end{lemma}
\begin{proof}
By Lemma \ref{L:n2=4}, there are two vertices fixed by all the involutions.  Without loss of generality, say the vertices are $v, v' \in V$.  Let $g$ be an element of odd order, and $\f$ be an involution; then $g\inv\f g$ is also an involution.  Hence $g\inv\f g(v) = v$, so $\f(g(v)) = g(v)$.  By Lemma \ref{L:f(V)=V}, $g(v) \in V$, so $g(v) \in \{v,v'\}$.  Since $g$ has odd order, this means $g$ fixes $v$ and $v'$ pointwise.  Hence $v$ and $v'$ are fixed by every element of $G$.
\end{proof}

\medskip
Burnside's Lemma \cite{bu} is a powerful tool for determining the possible numbers of fixed points of an automorphism of $K_{n,n}$.  Suppose that $G$ is a group acting faithfully on a graph $\gamma$, and let $U$ be a set of vertices preserved (setwise) by $G$.  Let  $|\fix(g|U)|$ denote the number of vertices of $U$ which are fixed by $g\in G$.  Burnside's Lemma gives us the following equation:

$${\rm \#\  vertex\  orbits} = \frac{1}{|G|}\sum_{g \in G}{|\mathrm{fix}(g|U)|}$$

The power of this equation lies in knowing that the left hand side must be an integer.  Say that $G = \TSG(\Gamma)$ is isomorphic to $A_4$ or $A_5$.  The Fixed Vertex Property and Lemma \ref{L:f(V)=V} imply that all elements of $G$ of the same order fix the same number of elements of $V$ and $W$.  We will let $n_i^v$ and $n_i^w$ denote the number of vertices fixed in each set by the elements of order $i$ in $A_4$, $S_4$ and $A_5$, with the exception that in $S_4$ $n_2^v$ and $n_2^w$ denote the number of vertices fixed by the involutions in $A_4 \leq S_4$, and we let $m_2^v$ and $m_2^w$ denote the number of vertices fixed by the involutions in $S_4 - A_4$.

We will first apply Burnside's Lemma to the action of a group isomorphic to $A_4$ on the vertex set $V$ of the graph $K_{n,n}$.  In this case, the equation becomes:
$${\rm \#\  vertex\  orbits} = \frac{1}{12}(n + 3n_2^v + 8n_3^v) \in \Z$$
\noindent The first term on the right comes from the identity, which fixes all $n$ vertices of $V$.

\begin{lemma} \label{L:n2=4l}
Let $G\leq \Aut(K_{n,n})$ which is isomorphic to $A_4$.  Suppose there is an embedding $\Gamma$ of $K_{n,n}$ in $S^3$ such that $G$ is induced on $\Gamma$ by an isomorphic subgroup $\widehat{G}\leq\so(4)$.   If $n_2^w = 0$, then $4 \vert n_2^v$.  
\end{lemma}
\begin{proof}
By the Automorphism Theorem, if $n_2^w = 0$, then $n_2^v = 2m$ for some integer $m$.  Then the total number of vertices fixed by an involution is $n_2 = n_2^v + n_2^w = 2m$.  Similarly, we let $n_3 = n_3^v+n_3^w$ denote the total number of vertices fixed by an element of order 3.  Applying Burnside's Lemma to both the action of $G$ on $V$ and the action of $G$ on $V \cup W$, we find that $\frac{1}{12}(n + 3(2m) + 8n_3^v) \in \Z$ and $\frac{1}{12}(2n + 3(2m) + 8n_3) \in \Z$.  Hence 12 must divide $2(n + 6m + 8n_3^v) - (2n + 6m + 8n_3) = 6m + 8(2n_3^v - n_3)$.  But, by the Automorphism Theorem, either $2n_3^v = n_3$ or $n_3^v = n_3 = 3k$.  In either case, 3 divides $2n_3^v - n_3$, and hence 12 divides $8(2n_3^v - n_3)$.  Therefore, 12 divides $6m$, so $m$ must be even.  So $n_2^v = 4l$ for some integer $l$.
\end{proof}

\begin{thm} \label{T:A4necessity}
If a complete bipartite graph $K_{n,n}$ has an embedding $\Gamma$ in $S^3$ such that $A_4 \leq \TSG(\Gamma)$, then $n \equiv 0, 2, 4, 6, 8 \pmod{12}$.
\end{thm}
\begin{proof}
Let $G \leq \TSG(\Gamma)$ such that $G \cong A_4$.  By the Isometry Theorem, there is a re-embedding $\Delta$ of $K_{n,n}$ such that $G$ is induced on $\Delta$ by an isomorphic subgroup $\widehat{G} \leq \so(4)$.  Applying Burnside's Lemma to the action of $G$ on the set $V$ (which is fixed by Lemma \ref{L:f(V)=V}), we find that $\frac{1}{12}(n + 3n_2^v + 8n_3^v) \in \Z$.  So, given values of $n_2^v$ and $n_3^v$, we can determine $n$, modulo 12.

By the Automorphism Theorem, either $n_2^v = n_2^w = 1$ or $2$, or $n_2^v = 2m$ and $n_2^w = 0$.  Similarly, either $n_3^v = n_3^w = 1$ or $2$, or $n_3^v = 3k$ and $n_3^w = 0$.  By Lemma \ref{L:n2} we know that $n_2^v \neq 1$, and by Lemma \ref{L:n3>2} we know that if $n_2^v = n_2^w = 2$, then $n_3^v \neq 1$.  By Lemma \ref{L:n2=4l}, we know if $n_2^w = 0$ then $n_2^v = 4l$ for some integer $l$.

The possible cases are summarized in the table below, along with the resulting value of $n$ (in the first two rows, $n_2^w = 2$; in the last three $n_2^w = 0$).

\begin{center}
\begin{tabular}{|c|c|c|}
	\hline
	$n_2^v$ & $n_3^v$ & $n\pmod{12}$ \\ \hline
	2 & 2 & 2 \\ \hline
	2 & $3k$ & 6 \\ \hline
	$4l$ & 1 & 4 \\ \hline
	$4l$ & 2 & 8 \\ \hline
	$4l$ & $3k$ & 0 \\ \hline
\end{tabular}
\end{center}

We conclude that $n \equiv 0, 2, 4, 6, 8 \pmod{12}$.
\end{proof}

\bigskip

The restrictions in Theorem \ref{T:A4necessity} hold whenever the topological symmetry group contains a subgroup isomorphic to $A_4$, $S_4$ or $A_5$.  However, if it contains a subgroup isomorphic to $S_4$ or $A_5$, we can find additional restrictions.  We first look at $S_4$.

\begin{lemma} \label{L:TwoFixed}
Let $G\leq \Aut(K_{n,n})$ which is isomorphic to $S_4$, and let $H \leq G$ be the subgroup isomorphic to $A_4$.  Suppose there is an embedding $\Gamma$ of $K_{n,n}$ in $S^3$ such that $G$ is induced on $\Gamma$ by an isomorphic subgroup $\widehat{G}\leq\so(4)$.   Assume the involutions of $H$ each fix exactly two vertices of $V$ and two vertices of $W$.  Then the action of $G$ fixes $V$ and $W$ setwise.
\end{lemma}
\begin{proof}
By Lemma \ref{L:n2=4}, the axes of rotation for all three involutions in $\widehat{H}$ meet in two points $v, v'$.  Let $g \in G$, and $\phi$ be an involution in $H$.  Then $g\inv\phi g$ is also an involution in $H$, so $g\inv \phi g(v) = v$.  Hence $\phi(g(v)) = g(v)$, so $g(v)$ is one of the four vertices fixed by $\phi$.  Since this is true for any involution in $H$, $g(v)$ must be fixed by all three involutions, and hence $g(v) \in \{v,v'\}$.  This means $g$ does not interchange $V$ and $W$, so by Fact \ref{L:automorphisms} it fixes $V$ and $W$ setwise.
\end{proof}

\begin{lemma} \label{L:n=6}
Let $G\leq \Aut(K_{n,n})$ which is isomorphic to $S_4$, and let $H \leq G$ be the subgroup isomorphic to $A_4$.  Then there does not exist an embedding $\Gamma$ of $K_{6,6}$ in $S^3$ such that $G$ is induced on $\Gamma$ by an isomorphic subgroup $\widehat{G}\leq\so(4)$.
\end{lemma}
\begin{proof}
Assume, towards contradiction, that there is such an embedding of $K_{6,6}$.  From the proof of Theorem \ref{T:A4necessity}, we know that $n_2^v = n_2^w = 2$, $n_3^v = 3k$ and $n_3^w = 0$.  From Lemma \ref{L:n2=4} and Lemma \ref{L:n3>2}, we know there are vertices $v$ and $v'$ which are fixed by the action of $H$, so $k \geq 1$.  Hence the elements of $\widehat{G}$ of order 3 are rotations whose axes each contain either 3 or 6 vertices of $V$.  Since two such axes can intersect in at most 2 points, and they all contain $v$ and $v'$, none of them can contain all 6 vertices of $V$.  So each axis of order 3 contains 3 vertices of $V$, and $k = 1$.  Moreover, each vertex of $V$ (other than $v$ and $v'$) lies on exactly one axis of order 3.

By Lemma \ref{L:TwoFixed}, the action of $G$ fixes $V$ and $W$ setwise; the Automorphism Theorem then implies that either $n_4^v = n_4^w = 2$, or $n_4^v = n_4^w = 0$ and the elements of order four induce a 2-cycle in each of $V$ and $W$.  We will consider each of these two cases.

\medskip

\noindent{\sc Case 1:}  We first assume $n_4^v = n_4^w = 2$.  If $g \in \widehat{G}$ has order 4, this means $g$ is a rotation.  Moreover, $g^2$ is an involution in $\widehat{H}$, so $v$ and $v'$ lie on the axis of this rotation; hence $g$ also fixes $v$ and $v'$.  Since $G$ is generated by the elements of orders 3 and 4, this implies $v$ and $v'$ are fixed by all of $G$.

Applying Burnside's Lemma to the action of $G$ on $V$ we have:
$${\rm \#\ of\ vertex\ orbits} = \frac{1}{24}(6 + 3n_2^v + 8n_3^v + 6m_2^v + 6n_4^v)$$
$$= \frac{1}{24}(6 + 3(2) + 8(3) + 6m_2^v + 6(2)) = \frac{1}{24}(48 + 6m_2^v) = 2 + \frac{m_2^v}{4} \in \Z$$

Hence $4\vert m_2^v$.  Since we know that $2 \leq m_2^v \leq 6$, this means $m_2^v = 4$.  Let $\psi$ be an involution in $G - H$, so $\psi$ fixes 4 vertices of $V$, including $v$ and $v'$.  Then $\psi$ is a rotation, with 4 vertices of $V$ on its axis of rotation.  There are 6 such involutions, but only 4 vertices of $V$ other than $v$ and $v'$.  So at least two of these involutions must fix at least three vertices of $V$ in common.  This means that two of these involutions have the same axis of rotation, which is a contradiction.

\medskip

\noindent{\sc Case 2:}  Now assume that $n_4^v = n_4^w = 0$.  Applying Burnside's Lemma to the action of $G$ on $V$ we have:
$${\rm \#\ of\ vertex\ orbits} = \frac{1}{24}(6 + 3(2) + 8(3) + 6m_2^v + 6(0))$$
$$= \frac{1}{24}(36 + 6m_2^v) = \frac{3}{2} + \frac{m_2^v}{4} \in \Z$$

Hence $m_2^v \equiv 2 \pmod{4}$, which means $m_2^v = 2$ or $6$.  If $m_2^v = 6$, the axis of an involution in $\widehat{G} - \widehat{H}$ will intersect the axis of each rotation of order 3 in three points, which is impossible.  So $m_2^v = 2$; similarly, $m_2^w = 2$.  Let $\psi$ be an involution in $G - H$.  Then there is an element $g \in H$ of order 3 such that $\psi g$ has order 4.  So $\psi g$ does not fix any vertices; since $g$ fixes $v$ and $v'$, this means $\psi$ does {\it not} fix either $v$ or $v'$.  So it must fix two other vertices in $V$.  Each of these vertices lies on an axis of rotation of order 3; since two distinct axes must intersect in either 0 or 2 points, the axis of $\psi$ must intersect the axis of rotation of order 3 in another point.  This point of intersection cannot be in the interior of an edge, or there would be three edges meeting at that point, and we would not have an embedding.  It cannot be at a vertex of $W$, since $n_3^w = 0$.  So both the vertices of $V$ fixed by $\psi$ must lie on the same axis of order 3.  But then this axis contains 4 vertices of $V$:  $v$, $v'$, and the two vertices shared with $\psi$.  This is again a contradiction.

\medskip
Therefore, there is no embedding $\Gamma$ of $K_{6,6}$ in $S^3$ such that $G$ is induced on $\Gamma$ by an isomorphic subgroup $\widehat{G}\leq\so(4)$.
\end{proof}

\bigskip

We now consider the additional restrictions that arise when the topological symmetry group has a subgroup isomorphic to $A_5$.  We first describe the subgroups of $\so(4)$ isomorphic to $A_4$ or $A_5$, from Du Val's classification of the subgroups of $\so(4)$.

\begin{fact} \label{F:A4A5structure} \cite{du}
The subgroups of $\so(4)$ isomorphic to $A_4$ and $A_5$ can be described as follows:
\begin{itemize}
	\item The only subgroup of $\so(4)$ isomorphic to $A_4$ is the group of symmetries of a solid regular tetrahedron (so all motions are rotations).
	\item The only subgroups of $\so(4)$ isomorphic to $A_5$ are the group of symmetries of a solid regular dodecahedron (in which all motions are rotations) and the group of symmetries of the 1-skeleton of a regular 4-simplex (in which the elements of order 5 are glide rotations, and the other elements are rotations).
\end{itemize}
\end{fact}

\medskip
In particular, in any group of motions isomorphic to $A_4$ or $A_5$, the elements of orders 2 and 3 are always rotations.  We will also use the following observation about the sizes of the orbits of points of $S^3$ under the action of the group of symmetries of a solid regular dodecahedron.

\begin{fact} \label{F:A5orbits}
If a group $G \cong A_5$ acts on $S^3$ as the group of symmetries of a solid regular dodecahedron, then every point $p$ in $S^3$ has an orbit of size 60 (if $p$ is not fixed by any element of $G$), 30 (if $p$ is fixed by an element of order 2), 20 (if $p$ is fixed by an element of order 3), 12 (if $p$ is fixed by an element of order 5) or 1 (if $p$ is fixed by rotations with distinct axes, in which case it is at the center of the dodecahedron or the center of its complement).
\end{fact}

\medskip
We apply these facts to obtain restrictions on when $\TSG(\Gamma) \cong A_5$ for an embedding $\Gamma$ of a complete bipartite graph.

\begin{lemma} \label{L:nvneq1}
Let $G\leq \Aut(K_{n,n})$ which is isomorphic to $A_5$.  Suppose there is an embedding $\Gamma$ of $K_{n,n}$ in $S^3$ such that $G$ is induced on $\Gamma$ by an isomorphic subgroup $\widehat{G}\leq\so(4)$.  Then an element of $G$ cannot fix exactly one vertex of $V$.
\end{lemma}
\begin{proof}
By Lemma \ref{L:n2}, no element of $G$ of order 2 can fix exactly one vertex (since any element of order 2 is also an element of a subgroup of $G$ isomorphic to $A_4$).

By Fact \ref{F:A4A5structure}, $\widehat{G}$ is either the group of symmetries of a regular solid dodecahedron, or the group of symmetries of a regular 4-simplex.  We will first consider the case when $\widehat{G}$ is the group of symmetries of a dodecahedron.  Assume that an element of $G$ of odd order $r$ ($r$ is 3 or 5) fixes exactly one vertex of $V$.  By the Fixed Vertex Property, every element of order $r$ will fix exactly one vertex of $V$.  By the Automorphism Theorem, the elements of order $r$ also fix exactly one vertex of $W$.  Let $A$ be an axis of rotation of order $r$, containing vertices $v \in V$ and $w \in W$.  Then $A$ is inverted by an element $\f$ of order 2.  Since the action of $G$ fixes $V$ and $W$ setwise, by Lemma \ref{L:f(V)=V}, $\f$ must fix $v$ and $w$, so these vertices must be at the center of the dodecahedron and the center of its complement.  But then all the rotations of order $r$ fix $v$ and $w$, and hence the edge between them.  Therefore all rotations of order $r$ share the same axis of rotation, which contradicts the assumption that $\widehat{G}$ is the group of symmetries of a solid regular dodecahedron.  So an element of odd order cannot fix exactly one vertex of $V$.

\medskip
We now consider the case when $\widehat{G}$ is the group of symmetries of a regular 4-simplex.  Then elements of $\widehat{G}$ of order 5 are glide rotations which have no fixed points.  It only remains to consider elements of order 3.  

Assume an element of $G$ of order 3 fixes exactly one vertex of $V$, and so (as above) every element of order 3 fixes exactly one vertex of $V$ and one vertex of $W$.  By Lemmas \ref{L:n2} and \ref{L:n3>2}, the involutions of $G$ cannot fix one or two vertices of each of $V$ and $W$.  So, by the Automorphism Theorem and the proof of Theorem \ref{T:A4necessity}, each involution fixes $4l$ vertices in $V$ and none in $W$, where $l \in \Z$.  Then, if we apply Burnside's Lemma to the action of $G$ on $W$ we find the number of vertex orbits is $\frac{1}{60}(n + 15(0) + 20(1) + 24(0)) = \frac{1}{60}(n + 20)$.  For this to be an integer, we must have $n = 60k + 40$ for some integer $k$, and there are $k+1$ orbits.  Since each orbit has size at most 60, there must be $k$ orbits of size 60 and one orbit of size 40.  But this is impossible, since 40 is not a factor of 60.

Therefore, the elements of order 3 in $G$ cannot fix exactly one vertex in $V$.
\end{proof}

\begin{thm} \label{T:A5necessity}
If a complete bipartite graph $K_{n,n}$ has an embedding $\Gamma$ in $S^3$ such that $A_5 \leq \TSG(\Gamma)$, then $n \equiv 0, 2, 12, 20, 30, 32, 42, 50 \pmod{60}$.
\end{thm}
\begin{proof}
Let $G \leq \TSG(\Gamma)$ such that $G \cong A_5$.  By the Isometry Theorem, there is a re-embedding $\Delta$ of $K_{n,n}$ such that $G$ is induced on $\Delta$ by an isomorphic subgroup $\widehat{G} \leq \so(4)$.  Recall that $n_i^v$ and $n_i^w$ denote the number of vertices in $V$ and $W$ fixed by the elements of $G$ of order $i$.  Applying Burnside's Lemma to the action of $G$ on the set $V$ (which is fixed setwise by Lemma \ref{L:f(V)=V}), we find that $\frac{1}{60}(n + 15n_2^v + 20n_3^v + 24n_5^v) \in \Z$.  So, given values of $n_2^v$, $n_3^v$ and $n_5^v$, we can determine $n$, modulo 60.

From the Automorphism Theorem and Lemma \ref{L:n2=4l}, we know that $n_2^v = 1$, $2$ or $4l$; $n_3^v = 1$, $2$ or $3k$; and $n_5^v = 1$, $2$ or $5m$.  By Lemma \ref{L:nvneq1}, $n_2^v \neq 1$, $n_3^v \neq 1$, and $n_5^v \neq 1$.  The table below summarizes the remaining possibilities, and gives the value of $n \pmod{60}$ in each case.

\begin{center}
\begin{tabular}{|c|c|c|c|}
	\hline
	$n_2^v$ & $n_3^v$ & $n_5^v$ & $n\pmod{60}$ \\ \hline
	2 & 2 & 2 & 2 \\ \hline
	2 & 2 & $5m$ & 50 \\ \hline
	2 & $3k$ & 2 & 42 \\ \hline
	2 & $3k$ & $5m$ & 30 \\ \hline
	$4l$ & 2 & 2 & 32 \\ \hline
	$4l$ & 2 & $5m$ & 20 \\ \hline
	$4l$ & $3k$ & 2 & 12 \\ \hline
	$4l$ & $3k$ & $5m$ & 0 \\ \hline
\end{tabular}
\end{center}
\end{proof}

\bigskip

We will now show that, in addition to Theorem \ref{T:A5necessity}, we also need to require that $n$ is larger than 30.  
 
\begin{lemma} \label{L:n3v=2}
Let $G\leq \Aut(K_{n,n})$ which is isomorphic to $A_5$.  Suppose there is an embedding $\Gamma$ of $K_{n,n}$ in $S^3$ such that $G$ is induced on $\Gamma$ by an isomorphic subgroup $\widehat{G}\leq\so(4)$.  If either every element of $G$ of order 3 fixes two vertices of $V$ and two vertices of $W$, or every element of $G$ of order 5 fixes two vertices of $V$ and two vertices of $W$, then there are two vertices in $V$ (or two in $W$) which are fixed by every element of $G$.
\end{lemma}
\begin{proof}
We first assume that every element of $G$ of order 5 fixes two vertices of $V$ and two vertices of $W$.  Since the elements of order 5 have non-empty fixed point sets, they are all induced by rotations, and hence $\widehat{G}$ is the group of symmetries of a regular solid dodecahedron, by Fact \ref{F:A4A5structure}.  This means all the axes of rotation intersect in two points.  Since each rotation of order 5 fixes two vertices from each of $V$ and $W$, it fixes a subgraph of $\Gamma$ isomorphic to $K_{2,2}$, which is homeomorphic to $S^1$, so the embedded subgraph is the axis of rotation.  Edges of $\Gamma$ may not intersect, so these axes must all intersect at two vertices.  The two vertices must both be in $V$ (without loss of generality), or all the elements of order 5 would fix a common edge, and hence not have distinct axes.  The two vertices are at the center of the solid dodecahedron, and the center of its complement, so they are fixed by every element of $G$.

Now assume that every element of $G$ of order 3 fixes two vertices of $V$ and two vertices of $W$.  By Fact \ref{F:A4A5structure}, each subgroup of $G$ isomorphic to $A_4$ is the group of symmetries of a regular solid tetrahedron.  This means all the axes of rotation in one such subgroup intersect in two points.  As in the previous paragraph, there are two vertices of $V$ (or two of $W$) fixed by every element of each subgroup.

We still need to show that different subgroups will fix the same pair of vertices.  Let $S_1$, $S_2$ and $S_3$ denote three of the five subgroups of $G$ isomorphic to $A_4$.  For each pair $i, j$, there is an element $\gamma_{ij}$ of order 3 such that $\gamma_{ij} \subseteq S_i \cap S_j$.  Assume $S_1$ and $S_2$ do not fix the same pair of vertices.  Then $\gamma_{12}$ must fix both pairs of vertices; hence one pair of vertices must be in $V$ and the other in $W$.  Say that $S_1$ fixes $v$ and $v'$ and $S_2$ fixes $w$ and $w'$.  So $\gamma_{13}$ fixes $v$ and $v'$, and $\gamma_{23}$ fixes $w$ and $w'$.  But then $S_3$ must fix one of these pairs; without loss of generality, $S_3$ fixes $v$ and $v'$, so $\gamma_{23}$ fixes $v$ and $v'$.  This means that $\gamma_{12}$ and $\gamma_{23}$ fix the same $K_{2,2}$ subgraph, and hence share the same axis, which is impossible, because no element of $A_5$ of order 3 is in three distinct subgroups isomorphic to $A_4$.

So all the subgroups must fix the same pair of vertices, both in $V$ or both in $W$.  Since these subgroups generate $A_5$, this means these two vertices are fixed by every element of $G$.
\end{proof}

\begin{lemma} \label{L:n>30}
If a complete bipartite graph $K_{n,n}$ has an embedding $\Gamma$ in $S^3$ such that $A_5 \leq \TSG(\Gamma)$, then $n > 30$.
\end{lemma}
\begin{proof}
Let $G \leq \TSG(\Gamma)$ such that $G \cong A_5$.  By the Isometry Theorem, there is a re-embedding $\Delta$ of $K_{n,n}$ such that $G$ is induced on $\Delta$ by an isomorphic subgroup $\widehat{G} \leq \so(4)$.  From Theorem \ref{T:A5necessity}, the only possibilities with $n \leq 30$ are $n = 2, 12, 20, 30$.  The automorphism group of $K_{2,2}$ is isomorphic to $D_4$, which does not have a subgroup isomorphic to $A_5$, so $n \neq 2$.

Assume $n = 12$.  Then, from Theorem \ref{T:A5necessity}, $n_5^v = n_5^w = 2$, $n_3^v = 3k$, $n_3^w = 0$, $n_2^v = 4l$ and $n_2^w = 0$.  Fact \ref{F:A4A5structure} implies that $\widehat{G}$ is the group of symmetries of a dodecahedron, and Lemma \ref{L:n3v=2} implies there are two vertices in $V$ which are fixed by every element of $G$.  These two vertices each have an orbit of size 1, and lie at the center of the dodecahedron and the center of its complement.  By Fact \ref{F:A5orbits}, all other points of $S^3$ have at least 12 points in their orbit under the action of $\widehat{G}$, but there are only 10 remaining vertices.  So $n \neq 12$.

Now assume $n = 20$.  Then, from Theorem \ref{T:A5necessity}, $n_3^v = n_3^w = 2$, $n_5^v = 5m$ and $n_2^v = 4l$.  By Lemma \ref{L:n3v=2}, this means there are two vertices in $V$ which are fixed by every element of $G$.  Therefore, the elements of $\widehat{G}$ are all rotations, which means $\widehat{G}$ is the group of symmetries of a dodecahedron.  So these two vertices each have an orbit of size 1, and lie at the center of the dodecahedron and the center of its complement.  By Fact \ref{F:A5orbits}, the remaining 18 vertices must be divided into orbits of sizes 12, 20, 30 or 60, which is impossible.  Hence $n \neq 20$.

Finally, we assume $n = 30$.  Then, from Theorem \ref{T:A5necessity}, $n_2^v = n_2^w = 2$, $n_3^v = 3k$, $n_3^w = 0$, $n_5^v = 5m$ and $n_5^w = 0$.  If $m \neq 0$, then the elements of order 5 are rotations, and $\widehat{G}$ induces the group of rotations of a regular solid dodecahedron by Fact \ref{F:A4A5structure}.  Then the involutions are each rotations with two vertices from $V$ and two vertices from $W$ on the axis of rotation; therefore, the entire axis is a subgraph of $\Gamma$ isomorphic to $K_{2,2}$.  Since edges of $\Gamma$ can only meet at vertices, and the axes of all the involutions meet at the center of the dodecahedron and the center of its complement, there must be vertices of $\Gamma$ at these two points.  Since $n_3^w = n_5^w = 0$, they are both vertices of $V$, and these two points each have an orbit of size 1.  By Fact \ref{F:A5orbits} the remaining 28 vertices of $V$ must be partitioned into orbits of size 12, 20, 30 or 60, which is impossible.  Therefore, we must have $m = 0$, and by Fact \ref{F:A4A5structure}, $\widehat{G}$ induces the rotations of the 1-skeleton of a regular 4-simplex in $S^3$.

If the vertices of the 4-simplex are denoted 1, 2, 3, 4 and 5, then each element of $\widehat{G}$ can be represented as an (even) permutation of these points.  Consider the involution $(12)(34)(5)$ (written in cycle notation).  The axis of rotation for this involution intersects the axes for $(13)(24)(5)$ and $(14)(23)(5)$ at the center of a regular tetrahedron (one of the 3-dimensional ``faces" of the 4-simplex) and the center of its complement.  It also intersects the axes for $(12)(35)(4)$ and $(12)(45)(3)$ at the midpoint of the edge $\overline{12}$ and the axes for $(34)(15)(2)$ and $(34)(25)(1)$ at the midpoint of $\overline{34}$.  Since each axis for an involution is a subgraph of $\Gamma$, each of the points of intersection must be a vertex of $\Gamma$.  Since the rotations of order 3 fix vertices of $V$ and none of $W$, then the points at the center of the tetrahedron and its complement are in $V$ and the points at the midpoints of the edges are in $W$.  The points at the midpoints of the edges have an orbit of size 10 under the action of $\widehat{G}$ (since the 4-simplex has 10 edges), so the vertices of $W$ are partitioned into at least two orbits.

However, by Burnside's Lemma, the number of vertex orbits of $W$ is $\frac{1}{60}(30 + 15(2) + 20(0) + 24(0)) = 1$.  This contradiction means that $n$ cannot be equal to 30.
\end{proof}

%%%%%%%%%%%%%%%%%%%%%%%%%%%%%%%%%%%

\section{Embeddings with $\TSG(\Gamma) \cong S_4$}

Since $A_4 \leq S_4$, Theorem \ref{T:A4necessity} tells us that if $K_{n,n}$ has an embedding $\Gamma$ with $\TSG(\Gamma) \cong S_4$, then $n \equiv 0, 2, 4, 6 {\rm \ or\ } 8 \pmod{12}$.  We can also restrict ourselves to $n \geq 4$, since otherwise the automorphism group of $K_{n,n}$ does not contain a subgroup isomorphic to $S_4$; and we know $n \neq 6$, by Lemma \ref{L:n=6}.  For each of these values of $n$, we will construct embeddings of $K_{n,n}$ whose topological symmetry group is isomorphic to $S_4$.  These constructions will make use of tools developed by Flapan, Naimi and the author in \cite{fmn1} and \cite{fmn2}:

\begin{edge} \cite{fmn2} 
Let $G$ be a finite subgroup of $\Diff(S^3)$ , and let $\g$ be a graph whose vertices are embedded in $S^3$ as a set $V$ which is invariant under $G$ such that $G$ induces a faithful action on $\g$.  Suppose that adjacent pairs of vertices in $V$ satisfy the following hypotheses:
\begin{enumerate}
	\item If a pair is pointwise fixed by non-trivial elements $h, g \in G$, then $\fix(h) = \fix(g)$.
	\item For each pair $\{v, w\}$ in the fixed point set $C$ of some non-trivial element of $G$, there is an arc $A_{vw} \subseteq C$ bounded by $\{v,w\}$ whose interior is disjoint from $V$ and from any other such arc $A_{v'w'}$ 
	\item If a point in the interior of some $A_{vw}$ or a pair $\{v,w\}$ bounding some $A_{vw}$ is setwise invariant under an $f \in G$, then $f(A_{vw}) = A_{vw}$.
	\item If a pair is interchanged by some $g \in G$, then the subgraph of $\g$ whose vertices are pointwise fixed by $g$ can be embedded in a proper subset of a circle.
	\item If a pair is interchanged by some $g \in G$, then $\fix(g)$ is non-empty, and $\fix(h) \neq \fix(g)$ if $h \neq g$.
\end{enumerate}
Then there is an embedding of the edges of $\g$ in $S^3$ such that the resulting embedding of $\g$ is setwise invariant under $G$.
\end{edge}

\begin{subgroup}
Let $\Gamma$ be an embedding of a 3-connected graph $S^3$, and let $H \leq \TSG(\Gamma)$.  Let $T$ be a set of edges in $\Gamma$ which is setwise invariant under $H$ and let $e \in T$.  Suppose that for any $\f \in \TSG(\Gamma)$ which pointwise fixes $e$ and setwise fixes the $H$-orbit of each edge in $T$, one of the following holds:
\begin{enumerate}
	\item \cite{fmn1} $\f$ also pointwise fixes a subgraph of $\Gamma$ that cannot be embedded in $S^1$, or
	\item there is a $\psi \in \TSG(\Gamma)$ such that $\fix(\f) \cap \fix(\psi)$ contains a pair of adjacent vertices but $\fix(\f)$ is not a subset of $\fix(\psi)$.
\end{enumerate}
Then there is an embedding $\Gamma'$ of $\Gamma$ with $H \cong \TSG(\Gamma')$.  Moreover, the embeddings $\Gamma$ and $\Gamma'$ agree on the vertices of the graph.
\end{subgroup}

The proof of the Subgroup Theorem when the first condition holds was given in \cite{fmn1}.  The proof when the second condition holds is very similar, and relies on the fact that the axes of two non-trivial rotations in $\so(4)$ are either the same, disjoint, or intersect in exactly two points.  In particular, the axes cannot intersect in only an arc.  We leave the details to the reader.

\bigskip

We will construct our embeddings by first embedding the vertices of $K_{n,n}$ so that there is a faithful action of $S_4$ on the vertices which satisfies the hypotheses of the Edge Embedding Lemma.  The Edge Embedding Lemma then allows us to embed the edges of the graph.  Finally, we use the Subgroup Theorem to show that the graph can be re-embedded so that the topological symmetry group is isomorphic to $S_4$.  In our initial embeddings of the vertices, we will use two different groups of isometries of $S^3$ which are isomorphic to $S_4$.  Let $G$ be the group of symmetries of a solid geodesic cube $\sigma$ (see Figure~\ref{F:S4Cube}), and let $H$ be the group of symmetries of the 1-skeleton of a regular geodesic tetrahedron $\tau$ (see Figure~\ref{F:S4}).  Also, let $E \leq G$ be the subgroup isomorphic to $A_4$ (so the involutions in $E$ are the squares of the rotations in $G$ of order 4) and $F \leq H$ be the subgroup isomorphic to $A_4$ (so $F$ is the group of symmetries of the solid regular tetrahedron $\tau$).  We let $s$ and $s'$ denote the center of the cube $\sigma$ and the center of its complement; and we let $t$ and $t'$ denote the center of the tetrahedron $\tau$ and the center of its complement.

\begin{figure} [h]
\scalebox{.7}{\includegraphics{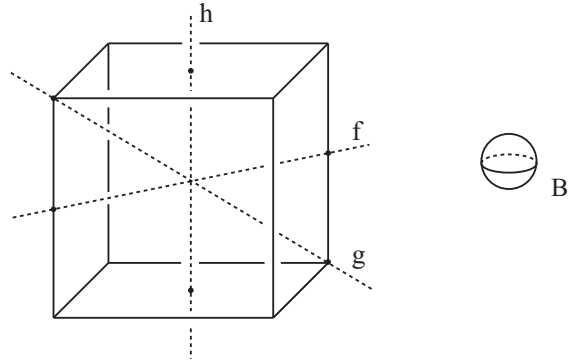}}
\caption{$G$ is the group of symmetries of a solid cube.  $f$, $g$, and $h$ are representative rotations around the axes shown of orders 2, 3 and 4, respectively; $f$ is a rotation of order 2 which is not the square of a rotation of order 4.}
\label{F:S4Cube}
\end{figure}

\begin{figure} [h]
\includegraphics{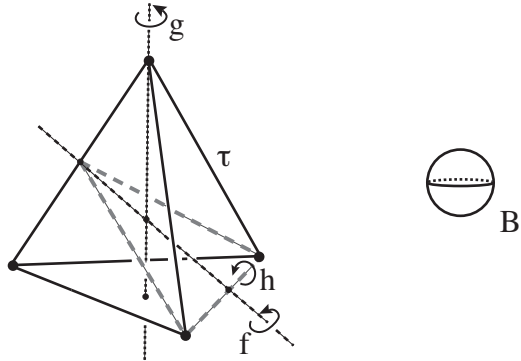}
\caption{$H$ leaves the 1-skeleton of a regular tetrahedron $\tau$ setwise invariant.  $g$ is a rotation of order 3, $f$ is a rotation of order 2 in $F$, and $h$ is a rotation of order 2 in $H - F$.}
\label{F:S4}
\end{figure}

\begin{prop} \label{P:S4Tetrahedron}
Let $n \equiv 0 {\rm \ or\ } 4 \pmod{12}$, and $n \geq 4$.  Then $K_{n,n}$ has an embedding $\Gamma$ with $S_4 \cong \TSG(\Gamma)$.
\end{prop}
\begin{proof}
We begin by constructing an embedding of $K_{12m, 12m}$.  Let $B$ be a ball which is in the complement of the fixed point sets of the elements of $H$, such that $B$ is disjoint from all its images under the action of $H$ (see Figure \ref{F:S4}).  Pick $m$ points of $B$.  The union of the orbits of these $m$ points under the action of $H$ contains $24m$ points.  Embed $m$ vertices of $V$ at each of the $m$ points in $B$, and embed the other vertices of $V$ as the images of these points under the action of $F \leq H$; denote this set $V_0$.  This embeds the $12m$ points of $V$.  Now embed the vertices of $W$ as the other $12m$ points in the orbits; denote this set $W_0$.

This gives an embedding of the vertices of $K_{12m,12m}$ upon which $H$ acts faithfully.  We check the conditions of the Edge Embedding Lemma.  Since no vertices are fixed by any element of $H$, conditions (1), (2), (3) and (4) are trivially satisfied.  The only elements of $H$ which interchange two adjacent vertices are the involutions in $H - F$.  These are each rotations with non-empty fixed point set, and no other element of $H$ shares the same fixed point set (since the squares of the elements of $H$ order 4 are involutions in $F$).  Hence condition (5) is satisfied.  Therefore, there is an embedding $\Gamma_0$ with $S_4 \leq \TSG(\Gamma_0)$.

Now we construct an embedding of $K_{12m+4,12m+4}$.  Consider two concentric geodesic tetrahedra $\tau_1$ and $\tau_2$ such that $\tau$ is between $\tau_1$ and $\tau_2$, so that the elements of $F \leq H$ preserve $\tau_1$ and $\tau_2$ and the elements of $H - F$ interchange them.  Let $V_4$ denote the four corners of $\tau_1$ and $W_4$ denote the four corners of $\tau_2$.  Embed the vertices of $V$ as $V_0 \cup V_4$ and the vertices of $W$ as $W_0 \cup W_4$.  Then $H$ acts faithfully on this embedding of $V \cup W$, and each element of $H$ induces an automorphism of $K_{12m+4,12m+4}$.  Observe that each element of $H$ of order 3 fixes one vertex from each of $V$ and $W$, and no other element of $H$ fixes any vertices (see Figure~\ref{F:K4S4}).

\begin{figure} [h]
\includegraphics{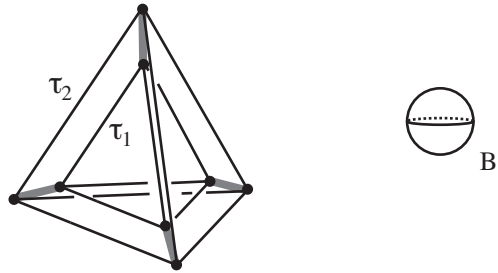}
\caption{The points of $V_4$ are the vertices of $\tau_1$ and the points of $W_4$ are the vertices of $\tau_2$.  The arcs required by hypothesis (2) are the gray arcs between corresponding vertices. }
\label{F:K4S4}
\end{figure}

If a pair of adjacent vertices is fixed by $h_1, h_2 \in H$, then $h_1$ and $h_2$ are elements of order 3 with the same axis of rotation, so condition (1)  of the Edge Embedding Lemma is satisfied.  Only one pair is fixed on each axis of rotation of order 3, and the vertices are connected by an arc which is disjoint from all other axes of rotation, so condition (2) is satisfied.  Any element which does not fix the arc pointwise sends it to a disjoint arc on a different axis of rotation, so condition (3) is satisfied.  The only elements of $H$ which interchange adjacent vertices are the involutions in $H - F$, which have no fixed vertices; so condition (4) is satisfied.  Finally, these involutions have fixed point sets homeomorphic to the circle, which are distinct from the fixed point sets of the other elements of $H$, so condition (5) is satisfied.  Therefore, there is an embedding $\Gamma_4$ with $S_4 \leq \TSG(\Gamma_4)$.

Now we apply the Subgroup Theorem to the embeddings $\Gamma_0$ and $\Gamma_4$.  First assume $m > 0$.  Let $v$ be a vertex of $V_0$, and let $w_1, w_2, w_3$ be vertices in $W_0$.  Let $e_i = \overline{v w_i}$.  Assume $\f \in \TSG(\Gamma_i)$ fixes $e_1$ pointwise, and fixes $\orb{e_i}_H$ setwise.  So $\f$ fixes $v$ and $w_1$.  Since $e_2$ is the only edge in $\orb{e_2}_H$ which is incident to $v$, $\f$ also fixes $e_2$ pointwise, and hence fixes $w_2$.  Similarly, $\f$ fixes $w_3$.  But then $\f$ fixes a subgraph isomorphic to $K_{1,3}$, which cannot be embedded in a circle.  So by the Subgroup Theorem, there is an embedding $\Gamma_i'$ with $S_4 \cong \TSG(\Gamma_i')$ (for $i = 0, 4$).

Finally, we consider $\Gamma_4$ when $m = 0$, so $n = 4$.  Let $v_1, v_2, v_3, v_4$ be the vertices of $V_4$, and let $w_1, w_2, w_3, w_4$ be the corresponding vertices of $W_4$, so that $v_i$ and $w_i$ lie on the same axis of rotation of order 3.  Now let $e_1 = \overline{v_1w_1}$ and $e_2 = \overline{v_1w_2}$.  Every edge of the embedded graph is in either $\orb{e_1}_H$ or $\orb{e_2}_H$.  Assume $\f \in \TSG(\Gamma_4)$ fixes $e_2$ pointwise and fixes each $\orb{e_i}_H$ setwise.  Then $\f$ fixes both $v_1$ and $w_2$.  Since $e_1$ is the only edge in $\orb{e_1}_H$ which is incident to $v_1$, $\f$ also fixes $e_1$, and hence $w_1$.  Moreover, $\overline{v_2w_2}$ is the only edge in $\orb{e_1}_H$ which is incident to $w_2$, so $\f$ will also fix this edge, and hence fix $v_2$.  So $\f$ fixes the 4-cycle $v_1w_1v_2w_2$ pointwise.  Let $\psi$ be the element of order 3 in $\TSG(\Gamma_4)$ which fixes $v_1$ and $w_1$ and permutes the other vertices cyclically.  Then $\fix(\f) \cap \fix(\psi)$ contains the adjacent pair $\{v_1, w_1\}$, but $\fix(\f)$ is not a subset of $\fix(\psi)$.  So by the Subgroup Theorem, there is an embedding $\Gamma_4'$ with $S_4 \cong \TSG(\Gamma_4')$.
\end{proof}

\bigskip

In the remaining cases, it will be easier to consider the values of $n$ modulo 24.

\begin{prop} \label{P:S4Cube}
Let $n \equiv 2, 6, 8, 14, 18 {\rm \ or\ } 20 \pmod{24}$, $n \geq 4$ and $n \neq 6$.  Then $K_{n,n}$ has an embedding $\Gamma$ with $S_4 \cong \TSG(\Gamma)$.
\end{prop}
\begin{proof}
We will first describe how to embed the vertices of $K_{n,n}$ in each case.  Then we will apply the Edge Embedding Lemma to embed the edges of the graph, and finally use the Subgroup Theorem to reduce the topological symmetry group to $S_4$.

We begin by embedding the vertices of $K_{24m, 24m}$.  Let $B$ be a ball which is in the complement of the fixed point sets of the elements of $G$, such that $B$ is disjoint from all its images under the action of $G$ (see Figure \ref{F:S4Cube}).  Pick $m$ points of $B$.  The union of the orbits of these $m$ points under the action of $G$ contains $24m$ points; embed the vertices of $V$ as these $24m$ points, and denote them $V_0$.  Similarly, pick another, disjoint, set of $m$ points in $B$ and embed the vertices of $W$ as the orbits of these points, denoted $W_0$.  Then $G$ acts faithfully on the set $V_0 \cup W_0$.  Note that none of the points of $V_0 \cup W_0$ are fixed by any of the elements of $G$, and no element of $G$ interchanges a vertex of $V_0$ with one of $W_0$.  We now consider each value of $n \pmod{24}$ in turn.

\medskip
\noindent{\sc Case 1:} $n \equiv 2 \pmod{24}$

Since $n > 2$, we have $n = 24m + 26$.  We will embed the vertices of $K_{24m+26, 24m+26}$.  Let $V_2$ consist of $s$ and $s'$ (the centers of $\sigma$ and its complement), along with a point of $B$ (not in $V_0 \cup W_0$) and its orbit under $G$, for a total of 26 points.  Let $W_2$ be the 26 points at the corners, the centers of the faces, and the midpoints of the edges of $\sigma$.  We embed $V$ as $V_0 \cup V_2$ and $W$ as $W_0 \cup W_2$.  Observe that every element of $G$ fixes two vertices of $V$ and two of $W$, alternating along the axis of rotation.

\medskip
\noindent{\sc Case 2:} $n \equiv 6 \pmod{24}$

Since $n \neq 6$, we have $n = 24m+30$.  We will embed the vertices of $K_{24m+30, 24m+30}$.  Consider three nested geodesic cubes $\sigma$, $\sigma_1$ and $\sigma_2$, with $\sigma$ being the middle cube; so all three cubes have center $s$.  Let $V_6$ be the set of points containing $s$, $s'$, the 12 points at the midpoints of the edges of $\sigma$, and the 16 points at the corners of $\sigma_1$ and $\sigma_2$.  Embed $V$ as the set of points $V_0 \cup V_6$.  Let $W_6$ be the 6 points at the centers of the faces of $\sigma$, along with a point of $B$ (not in $V_0 \cup W_0$) and its orbit under $G$, and embed $W$ as $W_0 \cup W_6$ (see Figure \ref{F:K6S4}).  Observe that each element of $G$ of order 3 fixes 6 vertices of $V$; each element of order 4 fixes two vertices of $V$ and two of $W$, alternating on the axis of rotation; each involution of $E \leq G$ fixes two vertices of $V$ and two of $W$, alternating on the axis of rotation; and each involution of $G - E$ fixes four vertices of $V$ and none of $W$.

\begin{figure} [h]
\scalebox{.6}{\includegraphics{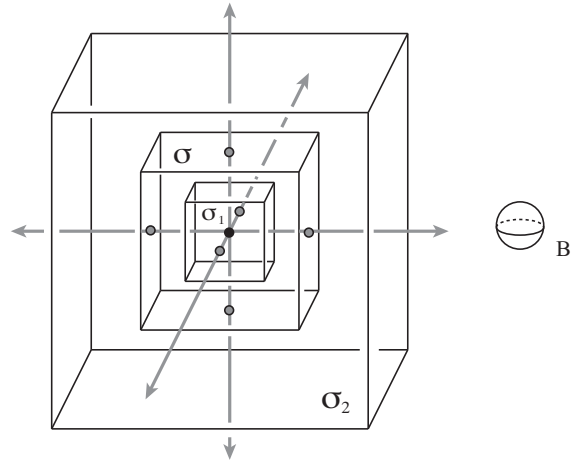}}
\caption{The points of $V_6$ are the center of the cube $s$ (the black dot), the center of its complement $s'$, the midpoints of the edges of $\sigma$ and the corners of $\sigma_1$ and $\sigma_2$.   The points of $W_6$ are the centers of the faces of $\sigma$ (the gray dots), along with 24 points in the orbit of $B$.  The arcs required by hypothesis (2) of the Edge Embedding Lemma are the gray arcs between the centers of the faces of $\sigma$ and the center of the cube and its complement. }
\label{F:K6S4}
\end{figure}

\medskip
\noindent{\sc Case 3:} $n \equiv 8 \pmod{24}$

We will embed the vertices of $K_{24m+8, 24m+8}$.  Let $V_8$ consist of $s$ and $s'$ together with the points at the centers of the faces of $\sigma$, and embed $V$ as $V_0 \cup V_8$.  Let $W_8$ be the 8 points at the corners of $\sigma$ and embed $W$ as $W_0 \cup W_8$ (see Figure \ref{F:K8S4}).  Observe that each element of $G$ of order 3 fixes two vertices of $V$ and two of $W$, alternating on the axis of rotation; each element of order 4 fixes four vertices of $V$ and none of $W$; each involution of $E \leq G$ fixes four vertices of $V$ and none of $W$; and each involution of $G - E$ fixes two vertices of $V$ and none of $W$. 

\begin{figure} [h]
\scalebox{.5}{\includegraphics{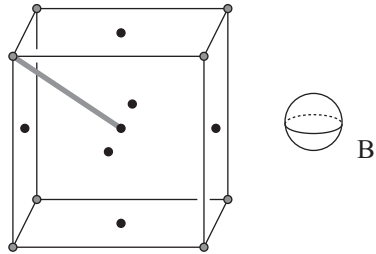}}
\caption{The points of $V_8$ are the black dots (and the point $s'$) and the points of $W_8$ are the gray dots.  The orbit of the gray line gives the arcs required by hypothesis (2) of the Edge Embedding Lemma. }
\label{F:K8S4}
\end{figure}

\medskip
\noindent{\sc Case 4:} $n \equiv 14 \pmod{24}$

We will embed the vertices of $K_{24m+14, 24m+14}$.  Let $V_{14}$ consist of $s$, $s'$, and the midpoints of the edges of $\sigma$; embed $V$ as $V_0 \cup V_{14}$.  Let $W_{14}$ be the corners and the centers of the faces of $\sigma$ and embed $W$ as $W_0 \cup W_{14}$ (see Figure \ref{F:K14S4}).  Observe that each element of $G$ of order 3 fixes two vertices of $V$ and two of $W$, alternating on the axis of rotation; each element of order 4 fixes two vertices of $V$ and two of $W$, alternating on the axis of rotation; each involution of $E \leq G$ fixes two vertices of $V$ and two of $W$, alternating on the axis of rotation; and each involution of $G - E$ fixes four vertices of $V$ and none of $W$.  

\begin{figure} [h]
\scalebox{.5}{\includegraphics{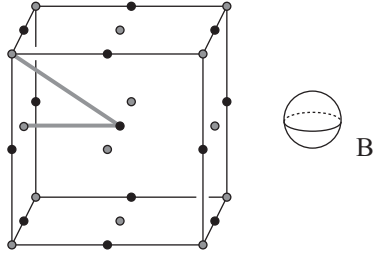}}
\caption{The points of $V_{14}$ are the black dots (and the point $s'$) and the points of $W_{14}$ are the gray dots.  The orbits of the gray lines give the arcs required by hypothesis (2) of the Edge Embedding Lemma. }
\label{F:K14S4}
\end{figure}

\medskip
\noindent{\sc Case 5:} $n \equiv 18 \pmod{24}$

We will embed the vertices of $K_{24m+18, 24m+18}$.  As for $K_{24m+30, 24m+30}$, consider three nested geodesic cubes $\sigma$, $\sigma_1$ and $\sigma_2$ with common center $s$, with $\sigma$ being the middle cube (see Figure \ref{F:K6S4}).  Let $V_{18}$ consist of $s$ and $s'$, together with the corners of $\sigma_1$ and $\sigma_2$.  Embed $V$ as $V_0 \cup V_{18}$.  Let $W_{18}$ be the midpoints of the edges and faces of $\sigma$; embed $W$ as $W_0 \cup W_{18}$.  Observe that each element of $G$ of order 3 fixes 6 vertices of $V$ and no vertices of $W$; each element of order 4 fixes two vertices of $V$ and two of $W$, alternating on the axis of rotation; each involution of $E \leq G$ fixes two vertices of $V$ and two of $W$, alternating on the axis of rotation; and each involution of $G - E$ fixes two vertices of $V$ and two of $W$, alternating on the axis of rotation.

\medskip
\noindent{\sc Case 6:} $n \equiv 20 \pmod{24}$

Finally, we embed the vertices $K_{24m+20, 24m+20}$.  Once again, consider three nested geodesic cubes $\sigma$, $\sigma_1$ and $\sigma_2$ centered at $s$, with $\sigma$ being the middle cube (see Figure \ref{F:K6S4}).  Let $V_{20}$ consist of $s$, $s'$, and the centers of the faces of all three cubes.  Embed $V$ as $V_0 \cup V_{20}$.  Let $W_{20}$ be the corners and the midpoints of the edges of $\sigma$; embed $W$ as $W_0 \cup W_{20}$.  Observe that each element of $G$ of order 3 fixes two vertices of $V$ and two of $W$, alternating on the axis of rotation; each element of order 4 fixes 8 vertices of $V$ and none of $W$; each involution of $E \leq G$ fixes 8 vertices of $V$ and none of $W$; and each involution of $G - E$ fixes two vertices of $V$ and two of $W$, alternating on the axis of rotation. 

\medskip
Now we apply the Edge Embedding Lemma to these embeddings.  In each case, for $k = 2, 6, 8, 14, 18, 20$, the arcs along the axes of rotation from the points of $W_k$ to the center of the cube and its complement are the arcs required by hypothesis (2) of the Edge Embedding Lemma; the remaining hypotheses are easy to check.  In particular, in every case the action of $G$ fixes $V$ and $W$ setwise, so hypotheses (4) and (5) are satisfied trivially.  Therefore, there is an embedding $\Gamma_k$ of the corresponding complete bipartite graph with $S_4 \leq \TSG(\Gamma_k)$.

\medskip
Our last step is to apply the Subgroup Theorem to each of our embeddings $\Gamma_k$.  In each case, if $m > 0$, then there is a vertex $v \in V_0$ embedded in $B$, which is not fixed by any element of $G$.  Let $w_1, w_2, w_3$ be any vertices in $W$.  Let $e_i = \overline{vw_i}$.  Assume $\f \in \TSG(\Gamma_k)$ fixes $e_1$ pointwise, and fixes $\orb{e_i}_G$ setwise.  So $\f$ fixes $v$ and $w_1$.  Since no element of $G$ fixes $v$ or interchanges $V$ and $W$, $e_2$ is the only edge in $\orb{e_2}_G$ which is incident to $v$.  So $\f$ also fixes $e_2$ pointwise, and hence fixes $w_2$.  Similarly, $\f$ fixes $w_3$.  But then $\f$ fixes a subgraph isomorphic to $K_{1,3}$, which cannot be embedded in a circle.  So by the Subgroup Theorem, there is an embedding $\Gamma_k'$ with $S_4 \cong \TSG(\Gamma_k')$.

So we may assume $m = 0$.  When $k = 2$ or $6$, there is still a vertex of $V_k$ or $W_k$ which is embedded in $B$, and we can use the same argument as the previous paragraph.  When $k = 8, 14, 18, 20$, two vertices of $V_k$ are embedded at $s$ and $s'$, which are fixed by every element of $G$.  Let $v_1$ be one of the other vertices of $V$, and $w_1$ be one of the vertices of $W$.  Let $e_1 = \overline{v_1w_1}$, $e_2 = \overline{sw_1}$ and $e_3 = \overline{s'w_1}$.  Assume $\f \in \TSG(\Gamma_k)$ fixes $e_1$ pointwise and fixes each $\orb{e_i}_G$ setwise.  Then $\f$ fixes both $v_1$ and $w_1$.  Since $s$ and $s'$ are fixed by $G$, $e_2$ and $e_3$ are the only edges in their respective orbits which are adjacent to $w_1$; so $\f$ must also fix these two edges pointwise.  Therefore $\f$ fixes a subgraph isomorphic to $K_{3,1}$, which cannot be embedded in $S^1$.  So by the Subgroup Theorem, there is an embedding $\Gamma_k'$ with $S_4 \cong \TSG(\Gamma_k')$.
\end{proof}

\medskip

Combining Theorem \ref{T:A4necessity}, Lemma \ref{L:n=6}, Proposition \ref{P:S4Tetrahedron} and Proposition \ref{P:S4Cube}, we obtain:

\begin{S4}
$K_{n,n}$ has an embedding $\Gamma$ with $S_4 \cong \TSG(\Gamma)$ if and only if $n \equiv 0, 2, 4, 6, 8 \pmod{12}$, $n \geq 4$, and $n \neq 6$.
\end{S4}

%%%%%%%%%%%%%%%%%%%%%%%%%%%%%%%%%%%%%%%%
\section{Embeddings with $\TSG \cong A_4$}

As in the previous section, let $F$ be the group of symmetries of a solid geodesic tetrahedron $\tau$ in $S^3$, so $F \cong A_4$.  We will first show that $K_{6,6}$ has an embedding with topological symmetry group isomorphic to $A_4$.

\begin{prop} \label{P:n=6}
$K_{6,6}$ has an embedding $\Gamma$ with $A_4 \cong \TSG(\Gamma)$.
\end{prop}
\begin{proof}
Let $V_6$ be the points $t$ and $t'$ at the center of $\tau$ and the center of its complement, together with the corners of $\tau$.  Let $W_6$ be the 6 points at the midpoints of the edges of $\tau$.  Embed $V$ as $V_6$ and $W$ as $W_6$.  Then $F$ acts faithfully on the vertices of $V$ and $W$.  Each element of order 3 fixes three vertices of $V$ and none of $W$, and each element of order 2 fixes two vertices of $V$ and two of $W$.  The arcs along the axes of rotation from the points of $W_{6}$ to the center of $\tau$ and its complement are the arcs required by hypothesis (2) of the Edge Embedding Lemma; the remaining hypotheses are easy to check.  Therefore, there is an embedding $\Gamma$ of $K_{6,6}$ with $A_4 \leq \TSG(\Gamma)$.

Now we apply the Subgroup Theorem.  Let $v$ be a vertex of $V$ embedded at one of the corners of $\tau$, and let $w$ be an embedded vertex of $W$.  Let $e_1 = \overline{wv}$, $e_2 = \overline{wt}$ and $e_3 = \overline{wt'}$.  Assume $\f \in \TSG(\Gamma)$ fixes $e_1$ pointwise and fixes each $\orb{e_i}_F$ setwise.  Then $\f$ fixes $w$ and $v$.  Since every element of $G$ fixes $t$ and $t'$, $e_2$ and $e_3$ are the only edges in their orbits which are also incident to $w$, so $\f$ also fixes both of these edges, and hence fixes $t$ and $t'$.  But $K_{3,1}$ cannot be embedded in a circle.  So by the Subgroup Theorem, there is an embedding $\Gamma'$ with $A_4 \cong \TSG(\Gamma')$.
\end{proof}

\medskip

For the remaining cases, we will apply the following corollary of the Subgroup Theorem to the embeddings described in Propositions \ref{P:S4Tetrahedron} and \ref{P:S4Cube}.

\begin{subgroupcor}\cite{fmn1} 
Let $\Gamma$ be an embedding of a 3-connected graph in $S^3$.  Suppose that $\Gamma$ contains an edge $e$ which is not pointwise fixed by any non-trivial element of $\TSG(\Gamma)$.  Then for every $H\leq \TSG(\Gamma)$, there is an embedding $\Gamma'$ of $\Gamma$ with $H \cong \TSG(\Gamma')$.
\end{subgroupcor}

\begin{prop} \label{P:A4embeddings}
Let $n \equiv 0, 2, 4, 6 {\rm \ or\ } 8 \pmod{12}$, $n \geq 4$.  Then $K_{n,n}$ has an embedding $\Gamma$ with $A_4 \cong \TSG(\Gamma)$.
\end{prop}
\begin{proof}
Consider the embeddings $\Gamma_0', \Gamma_2', \Gamma_4', \Gamma_6', \Gamma_8', \Gamma_{14}', \Gamma_{18}', \Gamma_{20}'$ from Propositions \ref{P:S4Tetrahedron} and \ref{P:S4Cube}.  Each of these embeddings has a topological symmetry group isomorphic to $S_4$ (induced by the action of the group $H$ on its vertices for $\Gamma_0'$ and $\Gamma_4'$, and by the action of the group $G$ on its vertices for the other graphs).  Moreover, for each $k$, $\Gamma_k'$ contains an edge which is not fixed by any non-trivial element of $\TSG(\Gamma_k')$.  The edge for each graph is listed in the table below.  

\begin{center}
\renewcommand{\arraystretch}{1.5}
\begin{tabular}{|c|c|}
	\hline
	{\bf Graph} & {\bf Edge not fixed by $\TSG(\Gamma_k')$} \\ \hline
	\rule{0ex}{2ex}$\Gamma_0'$ & \parbox[b]{3in}{\rule{0ex}{2ex}any edge of the graph} \\ \hline
	\rule{0ex}{2ex}$\Gamma_2'$ & \parbox[b]{3in}{\rule{0ex}{2ex}any edge adjacent to a vertex of $V_2$ other than the center of the cube or its complement}  \\ \hline
	$\Gamma_4'$ & \parbox[b]{3in}{\rule{0ex}{2ex}an edge from a corner of $\tau_1$ to a corner of $\tau_2$ which is not on an axis of rotation of order 3} \\ \hline
	$\Gamma_6'$ & \parbox[b]{3in}{\rule{0ex}{2ex}any edge adjacent to a vertex of $W_6$ not on the center of a face of the cube} \\ \hline
	$\Gamma_8'$ & \parbox[b]{3in}{\rule{0ex}{2ex}an edge from the center of a face of $\sigma$ to an adjacent corner of $\sigma$} \\ \hline
	$\Gamma_{14}'$ & \parbox[b]{3in}{\rule{0ex}{2ex}an edge from the midpoint of an edge of $\sigma$ to an adjacent corner of $\sigma$} \\ \hline
	$\Gamma_{18}'$ & \parbox[b]{3in}{\rule{0ex}{2ex}an edge from a corner of $\sigma_2$ to a midpoint of an edge of $\sigma$} \\ \hline
	$\Gamma_{20}'$ & \parbox[b]{3in}{\rule{0ex}{2ex}an edge from the center of a face of $\sigma$ to an adjacent corner of $\sigma$} \\ \hline
\end{tabular}
\end{center}

So by the Subgroup Corollary there is an re-embedding $\Gamma_k''$ of $\Gamma_k'$ with $A_4 \cong \TSG(\Gamma_k'')$.  This proves the result for $n \equiv 0, 2, 4, 6, 8 \pmod{12}$, $n \geq 4$, $n \neq 6$.  The final case, when $n = 6$, is given by Proposition \ref{P:n=6}.
\end{proof}

\medskip

Combining Theorem \ref{T:A4necessity} and Proposition \ref{P:A4embeddings}, we obtain:

\begin{A4}
$K_{n,n}$ has an embedding $\Gamma$ with $A_4 \cong \TSG(\Gamma)$ if and only if $n \equiv 0, 2, 4, 6 {\rm \ or\ } 8 \pmod{12}$ and $n \geq 4$.
\end{A4}

%%%%%%%%%%%%%%%%%%%%%%%%%%%%%%%%%%%%%%%%
\section{Embeddings with $\TSG \cong A_5$}

We now turn to the last of the polyhedral groups, $A_5$ (the group of symmetries of a regular dodecahedron).  By Theorem \ref{T:A5necessity} and Lemma \ref{L:n>30}, we know that if $K_{n,n}$ has an embedding $\Gamma$ with $\TSG(\Gamma) \cong A_5$, then $n \equiv 0, 2, 12, 20, 30, 32, 42 {\rm \ or\ } 50 \pmod{60}$, and $n > 30$.  We will construct embeddings for each of these possibilities.  Let $D$ be the group of symmetries of a regular solid dodecahedron $\Delta$ in $S^3$; then $D \cong A_5$.  Let $d$ denote the center of $\Delta$ and $d'$ denote the center of its complement; every element of $D$ fixes $d$ and $d'$.

\begin{prop} \label{P:A5Dodecahedron}
If $n \equiv 0, 2, 12, 20, 30, 32, 42 {\rm \ or\ } 50 \pmod{60}$, and $n > 30$, then $K_{n,n}$ has an embedding $\Gamma$ with $A_5 \cong \TSG(\Gamma)$.
\end{prop}
\begin{proof}
We will embed the vertices of $K_{n,n}$ for each case, and then apply the Edge Embedding Lemma and the Subgroup Theorem as in Propositions \ref{P:S4Tetrahedron} and \ref{P:S4Cube}.

\medskip
\noindent{\sc Case 1:} $n \equiv 0 \pmod{60}$

We will first embed the vertices of $K_{60m,60m}$.  Let $B$ be a ball which is in the complement of the fixed point sets of the elements of $D$, such that $B$ is disjoint from all its images under the action of $D$.  Pick $m$ points of $B$.  The union of the orbits of these $m$ points under the action of $D$ contains $60m$ points; denote this set of points by $V_0$ and embed the vertices of $V$ as $V_0$.  Similarly, pick another, disjoint, set of $m$ points in $B$ and embed the vertices of $W$ as the orbits of these points, denoted $W_0$.

\medskip
\noindent{\sc Case 2:} $n \equiv 2 \pmod{60}$

Since $n > 30$, this means $n = 60m + 62$.  We will embed the vertices of $K_{60m+62, 60m+62}$.  Let $V_2$ consist of $d$, $d'$, and a point of $B$ (not in $V_0 \cup W_0$) and its orbit under $D$.  Embed $V$ as $V_0 \cup V_2$.  Let $W_2$ consist of the 62 points at the corners, the centers of the faces, and the midpoints of the edges of  $\Delta$ and embed $W$ as $W_0 \cup W_2$.  Then $D$ acts faithfully on the embedded vertices.  Observe that every element of $D$ fixes two vertices of $V$ (namely, $d$ and $d'$) and two of $W$, alternating along the axis of rotation.

\medskip
\noindent{\sc Case 3:} $n \equiv 12 \pmod{60}$

Since $n > 30$, if $n \equiv 12 \pmod{60}$ then $n = 60m + 72$.  We will embed the vertices of $K_{60m+72, 60m+72}$.  Consider $\Delta$ and $\Delta'$, where $\Delta'$ is a second solid geodesic dodecahedron which is concentric with $\Delta$ (so $d$ is also the center of $\Delta'$).  Let $V_{12}$ consist of $d$, $d'$, the 50 points at the corners and midpoints of the edges of $\Delta$ and the 20 points at the corners of $\Delta'$.  Embed $V$ as $V_0 \cup V_{12}$.  Let $W_{12}$ consist of the 12 points at the centers of the faces of  $\Delta$ along with a point of $B$ (not in $V_0 \cup W_0$) and its orbit under $D$, and embed $W$ as $W_0 \cup W_{12}$.  Observe that each element of $G$ of order 5 fixes two vertices of $V$ and two of $W$, alternating on the axis of rotation; each element of order 3 fixes four vertices of $V$ and none of $W$; and each involution fixes four vertices of $V$ and none of $W$.

\medskip
\noindent{\sc Case 4:} $n \equiv 20 \pmod{60}$

Again, since $n > 30$, if $n \equiv 20 \pmod{60}$ then $n = 60m + 80$.  We will embed the vertices of $K_{60m+80, 60m+80}$.  Let $\Delta_1, \Delta_2, \Delta_3, \Delta_4$ be four concentric regular solid dodecahedra (all with center $d$ and group of symmetries $D$).  Let $V_{20}$ consist of $d$, $d'$, the 48 points at the centers of the faces of the $\Delta_i$, and the 30 points at the midpoints of the edges of $\Delta_1$.  Embed $V$ as $V_0 \cup V_{20}$.  Let $W_{20}$ be the 20 points at the corners of $\Delta_1$ along with a point of $B$ (not in $V_0 \cup W_0$) and its orbit under $D$, and embed $W$ as $W_0 \cup W_{20}$.  Observe that each element of $D$ of order 5 fixes 10 vertices of $V$ and none of $W$; each element of order 3 fixes two vertices of $V$ and two of $W$, alternating on the axis of rotation; and each involution fixes four vertices of $V$ and none of $W$.

\medskip
\noindent{\sc Case 5:} $n \equiv 30 \pmod{60}$

Once again, since $n > 30$, if $n \equiv 30 \pmod{60}$ then $n = 60m + 90$.  We will embed the vertices of $K_{60m+90, 60m+90}$.  As in the last paragraph, let $\Delta_1, \Delta_2, \Delta_3, \Delta_4$ be four concentric regular solid dodecahedra with center $d$.  Let $V_{20}$ consist of $d$, $d'$, the 48 points at the centers of the faces of the $\Delta_i$, and the 40 points at the corners of $\Delta_1$ and $\Delta_2$.  Embed $V$ as $V_0 \cup V_{30}$.  Let $W_{30}$ be the 30 points at the midpoints of the edges of $\Delta_1$ along with a point of $B$ (not in $V_0 \cup W_0$) and its orbit under $D$, and embed $W$ as $W_0 \cup W_{30}$.  Observe that each element of $D$ of order 5 fixes 10 vertices of $V$ and none of $W$; each element of order 3 fixes 6 vertices of $V$ and none of $W$; and each involution fixes two vertices of $V$ and two of $W$, alternating on the axis of rotation.

\medskip
\noindent{\sc Case 6:} $n \equiv 32 \pmod{60}$

We will embed the vertices of $K_{60m+32, 60m+32}$.  Let $V_{32}$ consist of $d$, $d'$, and the 30 points at the midpoints of the edges of $\Delta$.  Embed $V$ as $V_0 \cup V_{32}$.  Let $W_{32}$ be the 32 points at the corners and the centers of the faces of $\Delta$ and embed $W$ as $W_0 \cup W_{32}$.  Observe that each element of $D$ of order 5 fixes two vertices of $V$ and two of $W$, alternating on the axis of rotation; each element of order 3 fixes two vertices of $V$ and two of $W$, alternating on the axis of rotation; and each involution fixes four vertices of $V$ and none of $W$.

\medskip
\noindent{\sc Case 7:} $n \equiv 42 \pmod{60}$

We will embed the vertices of $K_{60m+42, 60m+42}$.  Let $\Delta_1$ and $\Delta_2$ be two concentric regular solid dodecahedra with common center $d$.  Let $V_{42}$ consist of $d$, $d'$, and the 40 points at the corners of $\Delta_1$ and $\Delta_2$.  Embed $V$ as $V_0 \cup V_{42}$.  Let $W_{42}$ be the 42 points at the centers of the faces and the midpoints of the edges of $\Delta_1$ and embed $W$ as $W_0 \cup W_{42}$.  Observe that each element of $D$ of order 5 fixes two vertices of $V$ and two of $W$, alternating on the axis of rotation; each element of order 3 fixes 6 vertices of $V$ and none of $W$; and each involution fixes two vertices of $V$ and two of $W$, alternating on the axis of rotation.

\medskip
\noindent{\sc Case 8:} $n \equiv 50 \pmod{60}$

Finally we will embed the vertices of $K_{60m+50, 60m+50}$.  Let $\Delta_1, \Delta_2, \Delta_3, \Delta_4$ be four concentric regular solid dodecahedra with common center $d$.  Let $V_{50}$ consist of $d$, $d'$, and the 48 points at the centers of the faces of the $\Delta_i$.  Embed $V$ as $V_0 \cup V_{50}$.  Let $W_{50}$ be the 50 points at the corners and the midpoints of the edges of $\Delta_1$ and embed $W$ as $W_0 \cup W_{50}$.  Observe that each element of $D$ of order 5 fixes 10 vertices of $V$ and none of $W$; each element of order 3 fixes two vertices of $V$ and two of $W$, alternating on the axis of rotation; and each involution fixes two vertices of $V$ and two of $W$, alternating on the axis of rotation.

\medskip
Now that we have embedded the vertices in every case, it only remains to use the Edge Embedding Lemma to embed the edges, and then the Subgroup Theorem to find an embedding whose topological symmetry group is isomorphic to $A_5$.  These arguments are almost identical to those in Proposition \ref{P:S4Cube}, and the details are left to the reader.
\end{proof}

\bigskip
Combining Theorem \ref{T:A5necessity}, Lemma \ref{L:n>30} and Proposition \ref{P:A5Dodecahedron}, we obtain:

\begin{A5} \label{T:A5}
$K_{n,n}$ has an embedding $\Gamma$ with $A_5 \cong \TSG(\Gamma)$ if and only if $n > 30$ and $n \equiv 0, 2, 12, 20, 30, 32, 42 {\rm \ or\ } 50 \pmod{60}$.
\end{A5}

\bigskip

\subsection*{Acknowledgments}  The author would like to thank Erica Flapan and Ramin Naimi for many helpful conversations, and for their comments on early drafts of this paper.

\small

\normalsize

\end{document}